\documentclass[12pt,epsfig,amsfonts]{amsart} 
\usepackage{amsmath,amsthm,amssymb,amscd,epsfig,color}
\usepackage{graphicx}
\usepackage{mathrsfs}

\newtheorem*{theorema}{Main Theorem}

\newtheorem{theorem}{Theorem}

\newtheorem{prop}{Proposition}[section]
\newtheorem{lemma}[prop]{Lemma}

\theoremstyle{definition}

\theoremstyle{remark}

\numberwithin{equation}{section}

\begin{document}

\author{Hiroki Takahasi}

\address{Keio Institute of Pure and Applied Sciences (KiPAS), Department of Mathematics,
Keio University, Yokohama,
223-8522, JAPAN} 
\email{hiroki@math.keio.ac.jp}
\urladdr{\texttt{http://www.math.keio.ac.jp/~hiroki/}}


%\thanks{
%}

\title[Large Deviations for
 denominators of continued fractions]
 {Large deviations for\\
 denominators of continued fractions} 
 \maketitle
 
 \begin{abstract}
 We give an exponential upper bound on the probability
 with which the denominator of the $n$th convergent in the regular continued fraction expansion
 stays away from the mean $\frac{n\pi^2}{12\log2}$.
  The exponential rate is best possible, given by an analytic function
 related to the dimension spectrum of Lyapunov exponents for the Gauss transformation.
  %We show that the
 %large deviation principle holds and the rate function is given by this analytic function.
 %We give an upper bound on the number of periodic points of period $n$ of the Gauss transformation
 % We show the analyticity of the rate function by relating it to the dimension spectrum of 
 %Lyapunov exponents of the Gauss map, and use this smoothness to
 %deduce asymptotic upper bounds.     
 \end{abstract}

%{\bf Put the graph of $I$.}

%\begin{theorema}
%$$\lim_{\epsilon\to0}\limsup_{n\to\infty}\frac{1}{n}\#\{x\in{\rm Per}_n(T)\colon
%|\chi(x)-\alpha|\leq\epsilon\}$$
%$$=\lim_{\epsilon\to0}\liminf_{n\to\infty}\frac{1}{n}\#\{x\in{\rm Per}_n(T)\colon
%|\frac{2}{n}\log q_n(x)-\alpha |\leq\epsilon\}=\alpha b(\alpha).$$
%$\alpha\mapsto \alpha b(\alpha)$ is strictly monotone increasing.
%\end{theorema}
%$\alpha b(\alpha)=\alpha-I(\alpha)$. The right-hand-side is the sum of two convex functions,
%and so $\alpha b(\alpha)$ is convex. To show the strict monotonicity, look at the graph of $I$.
    
    \section{Introduction}
    Each irrational number $x\in(0,1)$ has the continued fraction expansion
    \begin{equation*}\label{expansion}x=\cfrac{1}{a_1+\cfrac{1}{a_2+\cdots}},\end{equation*}
where each $a_i$ is a positive integer.
Let $p_n$, $q_n$ be relatively prime positive integers satisfying
$$\frac{p_n}{q_n}=\cfrac{1}{a_1+\cfrac{1}{a_2+\cdots+\cfrac{1}{a_n}}}.$$
Then $p_n/q_n$ converges to $x$ as $n\to\infty$, and 
the rate of this convergence is determined by
the growth rate of the denominator $q_n$:
\begin{equation}\label{double}
\frac{1}{2q_{n+1}^2}\leq\left|x-\frac{p_n}{q_n}\right|\leq\frac{1}{q_n^2}.\end{equation}
%See \cite{Krai} for example.
One important problem in %Diophantine approximations and 
the metric theory of continued fractions 
%ever since the time of Gauss 
is to investigate
the limit behavior of $q_n$
for typical irrationals.
It was Khinchin \cite{Khi35} who proved 
the existence of an absolute constant $\gamma$
such that $(1/n)\log q_n\to\gamma$ as $n\to\infty$ Lebesgue-a.e.
L\'evy \cite{Lev37} showed $\gamma=\frac{\pi^2}{12\log2}$.
Hence,
 for any closed interval $K$ not containing $\gamma$,
 the Lebesgue measure of the event
$\{\log q_n-\gamma n\in K\}$ converges to $0$ as $n\to\infty$.
Of interest to know is the rate of this convergence.
If it is exponential, namely there is an upper bound of the form $Ce^{-\delta n}$ for some $C>0$ and $\delta>0$,
then the smallest such $\delta$ would have some intrinsic meaning.

Some rates in this convergence are available from central limit theorems.
Denote by $\lambda$ the Lebesgue measure restricted to $(0,1)$.
Misevi\u{c}jus \cite{Mis81} showed that
$$\sup_{\alpha\in\mathbb R}\left|\lambda\left\{\frac{\log q_n-\gamma n}{\sigma\sqrt{n}}\leq\alpha\right\}
-\int_{-\infty}^\alpha e^{-\frac{x^2}{2}}dx
\right|
=O\left(\frac{\log n}{\sqrt{n}}\right),$$
where $\sigma>0$.
The results of Morita \cite{Mor94} and Vall\'ee \cite{Val97} improve
the order to $O(1/\sqrt{n})$.
It follows that for every $\alpha<2\gamma$,
$$\lambda\left\{\frac{2}{n}\log q_n\leq\alpha\right\}=O\left(\frac{1}{\sqrt{n}}\right),$$
%$$\lambda\left\{\frac{\log q_n}{ n}<\alpha\right\}\leq \int_{-\infty}^{\sqrt{n}\sigma^{-1}(-\pi^2/12\log2+\alpha)}e^{-\frac{x^2}{2}}dx
%+O\left(\frac{1}{\sqrt{n}}\right).$$
%For $\alpha<\pi^2/12\log2$, this is
%$O\left(\frac{1}{\sqrt{n}}\right).$
%$e^x\geq 1+x$
This estimate is far from optimal.
From the result of Ara\'ujo and Bufetov \cite[Theorem B]{AraBuf}, 
$$\limsup_{n\to\infty}\frac{1}{n}\log\lambda\left\{\frac{2}{n}\log q_n\leq\alpha\right\}\leq I(\alpha),$$
where the number $I(\alpha)>0$ is defined below. Hence, the convergence takes place at an exponential 
rate, and the rate can be chosen arbitrarily close to $I(\alpha)$.
The aim of this paper is to show that $I(\alpha)$ is the best exponential rate.

%give a better upper bound on
%$\lambda\left\{\frac{2}{n}\log q_n\leq\alpha\right\}$ with the best exponential rate.
%We give exponential bounds on these probabilities,  with rates given by a certain analytic function
%of $\alpha$.
% These digits are generated by iterating {\it the Gauss transformation} 
   %   $T\colon (0,1]\to[0,1)$ given by $Tx=1/x-\lfloor1/x\rfloor$ (mod 1).
      %A simple computation shows $a_i(x)=\left\lfloor1/T^{i-1}x\right\rfloor.$
 %see \cite{FLWW09}.
%From \eqref{double} and \eqref{relation} it follows that
 %$$\lim_{n\to\infty}\frac{1}{n}\log\left|x-\frac{p_n}{q_n}\right|^{-1}=\lim_{n\to\infty}\frac{2}{n} \log q_n(x)=
 %\lim_{n\to\infty}\frac{1}{n}\log|DT^n(x)|,$$
%if one of the limits exists.
%This indicates that, the large deviations of the Diophantine approximation
%translates to that of the time averages of the function $\log|DT|$.
%This number
%coincides with the exponential rate of continued fraction approximation of $x$
%(See Kesseb\"ohmer and Stratmann \cite[Proposition 1.2]{KesStr07}), i.e.,
%\begin{equation*}
%\lim_{n\to\infty}\frac{2}{n}\log q_n(x)=-\lim_{n\to\infty}\frac{2}{n}
%\log\left|x-\frac{p_n(x)}{q_n(x)}\right|,\end{equation*}
%whenever one of the limits exists.
We now define $I(\alpha)$ and state our main result.
For each $\alpha\in[0,\infty]$ define
%consider the set $L(\alpha)$ of points for which the Lyapunov exponent of $T$ is equal 
%to $\alpha$, namely
$$L(\alpha)=\left\{x\in (0,1)\colon\liminf_{n\to\infty}\frac{2}{n}\log q_n(x)
=\limsup_{n\to\infty}\frac{2}{n}\log q_n(x)
=\alpha\right\}.$$
Put $b(\alpha)=\dim_H L(\alpha),$
where $\dim_H$ denotes the Hausdorff dimension.
Define $$I(\alpha)=\alpha(1-b(\alpha)).$$
Put $\alpha_{\rm min}=\log\frac{\sqrt{5}+1}{2}$.

\begin{theorema}
The following holds:
\medskip

\noindent {$\bullet$} for every $n\geq1$ and every $\alpha\in(2\gamma+\frac{16}{n},\infty)$,
$$\lambda\left\{\frac{2}{n}\log q_n
 \geq\alpha \right\}\leq C_\alpha e^{-I(\alpha)n};$$
\noindent {$\bullet$} for every $n\geq1$ with
$\alpha_{\rm min}< 2\gamma-\frac{16}{n}$ and every $\alpha\in(\alpha_{\rm min},2\gamma-\frac{16}{n})$ ,
$$\lambda\left\{ \frac{2}{n}\log q_n\leq \alpha \right\}
\leq C_\alpha e^{-I(\alpha)n},$$
%For every interval $K$,
%$$\lambda\left\{-\frac{1}{n}\log\left|x-\frac{p_n}{q_n}\right|\in K\right\}\leq Ce^{-n\inf_{\alpha\in K}I(\alpha)}.???$$
where $C_\alpha:=e^{16(|I'(\alpha)|+1)}$.
\end{theorema}

The Main Theorem follows from a combination of 
 the multifractal analysis \cite{KesStr07,PolWei99} and the thermodynamic formalism \cite{Bow75,Rue78}
 associated with
the Gauss transformation
      $T\colon (0,1]\to[0,1)$ given by $Tx=1/x-\lfloor1/x\rfloor$.
        It is well-known (see e.g., \cite{Krai} or Lemmas \ref{length} and \ref{distortion}) that there exists a constant $C>1$ such that
for any irrational $x\in(0,1)$ and $n\geq1$,
\begin{equation*}\label{relation}C^{-1}q_n^2(x)\leq |DT^n(x)|\leq Cq_n^2(x).\end{equation*} 
These double inequalities permit to translate the analysis of $\log q_n$
to that of the Birkhoff sum of the function $\log|DT|$ under the iteration of $T$.
The $L(\alpha)$ is the set of irrationals in $(0,1)$ for which the Lyapunov exponent for $T$ is equal to $\alpha$. 
Then $L(\alpha)\neq\emptyset$ holds if and only if $\alpha\in\left[\alpha_{\rm min},\infty\right]$,
 see \cite{KesStr07,PolWei99}.
The function $\alpha\in     \left[\alpha_{\rm min},\infty\right)         \mapsto b(\alpha)$ is
known as {\it the dimension spectrum of Lyapunov exponents}.
It is a non-convex function, analytic on $\left(\alpha_{\rm min},\infty\right)$ \cite[Theorem 1.3]{KesStr07}, and 
$b(\alpha)=1$ holds if and only if $\alpha=2\gamma$.
Hence,  $\alpha\in     \left(\alpha_{\rm min},\infty\right)         \mapsto I(\alpha)$ is 
analytic and $I(\alpha)=0$ holds if and only if $\alpha=2\gamma$. 
%We show that $I$ is a convex function (Lemma \ref{ratefunction}).

The graph of the function  $\alpha\in[\alpha_{\rm min},\infty)\mapsto I(\alpha)$ is shown in FIGURE 1.
Since $b(\alpha_{\rm min})=0$ by \cite[Theorem 1.3]{KesStr07}, $I(\alpha_{\rm min})=\alpha_{\rm min}$ holds.
Since $I$ is convex by Lemma \ref{ratefunction}, $I'(\alpha)$ increases for $\alpha>2\gamma$.
Since
$b(\alpha)\to1/2$ as $\alpha\to\infty$ by \cite[Theorem 1.3]{KesStr07}, the asymptote exists with slope $1/2$.
% $I'(\alpha)\to1/2$ as $\alpha\to\infty$
%and 
%The slope of the asymptote is $1/2$.

\begin{figure}
\begin{center}
\includegraphics[height=4cm,width=7.5cm]{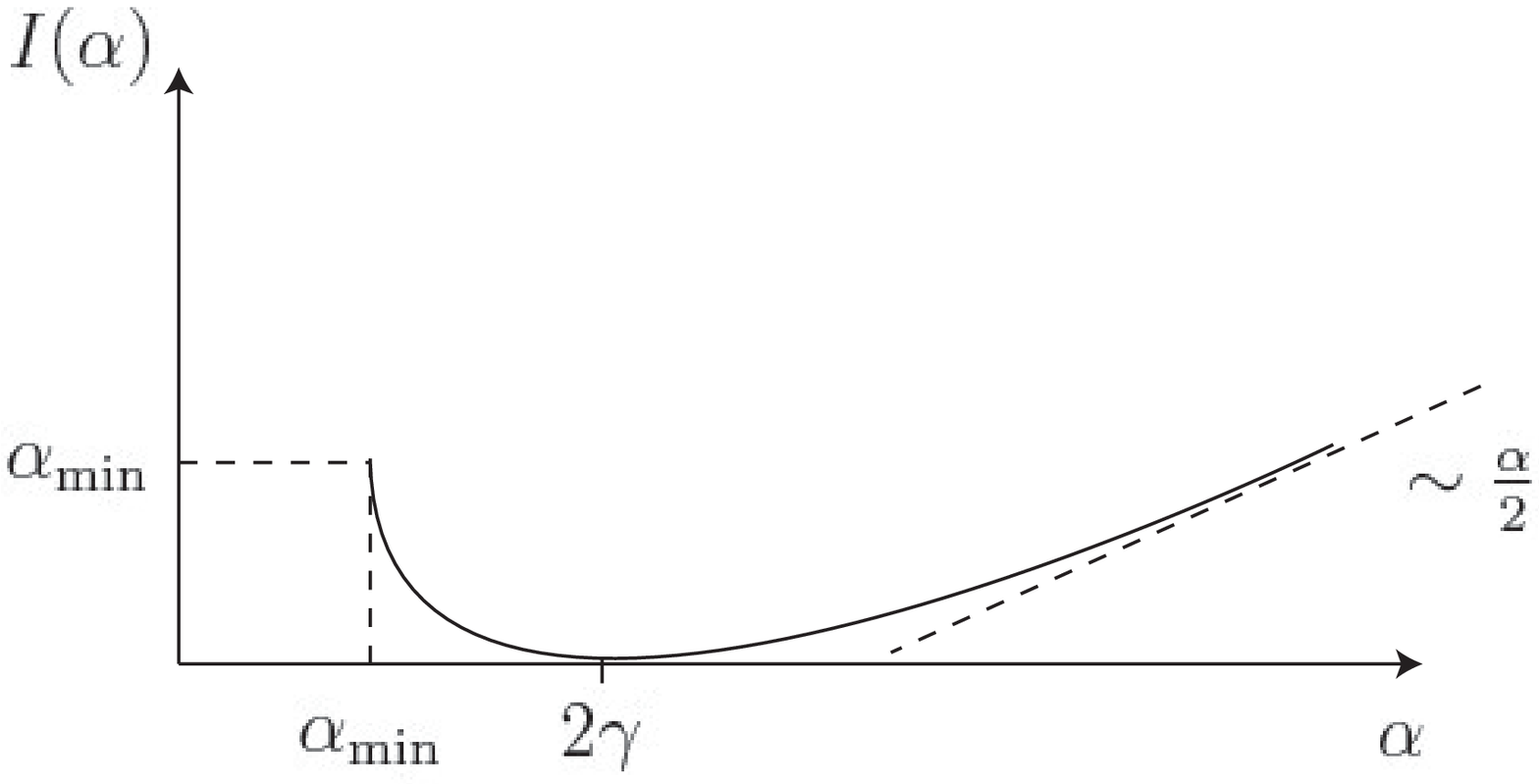}
\caption{The graph of the function $\alpha\in[\alpha_{\rm min},\infty)\mapsto I(\alpha)$: $I(\alpha_{\rm min})=\alpha_{\rm min}$, 
$\displaystyle{\lim_{\alpha\to\alpha_{\rm min}+0}}$$I'(\alpha)=-\infty$,  $\displaystyle{\lim_{\alpha\to\infty}}$$I'(\alpha)=1/2$.}
\end{center}
\end{figure}

%The chaotic features of $T$, the sequence $\log|DT\circ T^i|$ behave as if i.i.d. random variables.
Since $T$ has infinitely many branches and $\log|DT|$ is unbounded, some finite approximations
are necessary for a proof of the Main Theorem. We take finite subsystems, and estimate the exponent of 
the deviation probabilities %along 
% the classical line of the thermodynamic formalism \cite{Bow75,Rue78},
 in terms of entropy and Lyapunov exponents of invariant probability measures of $T$
 supported on the subsystems (Lemma \ref{horse}).
 %Our method,  inspired by the work of Takahashi \cite{Tak87},
 % solves the problem of unboundedness of functions and that of 
% the existence of infinitely many branches simultaneously.
Then, using the variational formula for the dimension spectrum 
 \cite{PolWei99} %$\alpha\in     \left(\log\frac{\sqrt{5}+1}{2},\infty\right)         \mapsto b(\alpha)$
we relate
the exponent to the function $I$.
At the very end we use
the convexity and the smoothness of $I$ (in fact, $C^2$ is sufficient)
to bound error terms arising from the nonlinearity of $T$ and deduce the desired upper bounds.
%To obtain upper bounds, we choose a finite subsystem and 
%invariant measures supported on it.
% Then, the analyticity of $I$ is a consequence of that of the dimension spectrum, as shown by
%Kesseb\"ohmer and Stratmann \cite[Theorem 1.3]{KesStr07}. 

From \cite[Theorem B]{AraBuf} the following asymptotic lower bounds hold:
\begin{itemize}
\item  for every $\alpha\in(2\gamma,\infty)$,
\begin{align*}
\liminf_{n\to\infty}&\frac{1}{n}\log \lambda\left\{\frac{2}{n}\log q_n\geq \alpha \right\}
\geq-I(\alpha);
\end{align*}
\item for every $\alpha\in(\alpha_{\rm min},  2\gamma)$,
\begin{align*}
\liminf_{n\to\infty}&\frac{1}{n}\log \lambda\left\{\frac{2}{n}\log q_n\leq \alpha \right\}
\geq-I(\alpha).
\end{align*}
\end{itemize}
This means that the exponent $I(\alpha)$ in the Main Theorem is the best possible one.
However,  the result below on sample means of
independent and identically distributed (i.i.d.) random variables 
leaves the possibility that the upper bounds in the Main Theorem
can be improved.
\begin{theorem}\cite[Theorem 1]{BahRan60}\label{th}
Let $(X_n)_{n\geq1}$ be a sequence of i.i.d. random variables with positive variance with mean $0$.
Assume the moment generating function $c(t)=\log E(e^{tX_1})$ is finite
on some interval $U$.
Let $\alpha>0$ and $t_\alpha\in U$ be such that $J(\alpha):=
\sup_{t\in U}e^{t\alpha-c(t)}=e^{t_\alpha\alpha-c(t_\alpha)}$.
Then for every $\alpha>0$,
%$$P(X_1+\cdots+X_n\geq\alpha n)\leq e^{-I(\alpha)n},$$
$$P(S_n\geq\alpha n)=\frac{b_n(1+o(1))}{\sqrt{2\pi n}}e^{-J(\alpha)n},$$
where $S_n=X_1+\cdots+X_n$, $(b_n)_{n\geq1}$ is a sequence of constants %\textcolor{red}{depending on $\alpha$ ?} 
and $\inf_n b_n>0$, $\sup_nb_n<\infty$.
\end{theorem}
In \cite{BahRan60} it was shown that $\frac{b_n(1+o(1))}{\sqrt{2\pi n}}\leq1$, and so
$P(S_n\geq\alpha n)\leq e^{-J(\alpha)n}$ holds.
Such an exponential upper bound was obtained in \cite{Che52}, and follows from Cram\'er's theorem on the LDP, see 
\cite[pp.26-27]{RasSep}. 
%Theorem A gives one indication of randomness of the non-i.i.d. sequence $\{\log|DT\circ T^n|\}_{n\geq1}$.
 %Owing to Theorem \ref{th}, it is hopeless to expect lower bounds of order $e^{-I(\alpha)n}$ in 
 %the Main Theorem.
In the non-i.i.d. case, results for uniformly hyperbolic systems on compact metric spaces
with H\"older continuous functions are available \cite[Lemma A.1]{ChaCol05}, \cite[Theorem 1]{Wad96}, which
provide 
upper and lower bounds in agreement with the i.i.d. case in Theorem \ref{th}.
The bounds in   \cite[Lemma A.1]{ChaCol05} are valid only for those $\alpha$
close to the mean. 

\section{Preliminary lemmas}
Before entering the proof of the Main Theorem we need some preliminary lemmas.
For each integer $n\geq1$
denote by $\mathscr{A}^n$ the collection of maximal open intervals on which $T^n$ is well-defined
and continuous. %=\vee_{i=0}^{n-1}T^{-i}\mathscr{A}.$
Notice that $q_n$ is constant on each element $A\in\mathscr{A}^n$.
This constant value is denoted by $q_n(A)$.
For a finite set $\mathscr{B}$ of $\mathscr{A}^n$ denote by $[\mathscr{B}]$ the union of all its elements.

\begin{lemma}\label{length}
For every integer $n\geq1$ and every $A\in \mathscr{A}^n$,
\begin{equation*}
\frac{1}{2}\leq
\frac{\lambda(A)}{q_n(A)^{-2}}< 1.
%\lambda(A)<q_n(A)^{-2}.
\end{equation*}
\end{lemma}

\begin{proof}
Assume $n=1$. Each $A\in\mathscr{A}^1$ has the form
$A=(1/(k+1),1/k)$ for some $k\geq1$.
Then $q_1(A)=k$ and so the double inequalities hold.
Assume $n\geq2$ and let $A\in\mathscr{A}^n$.
For each $i=1,\ldots,n$, $p_i$, $q_i$ are constant on $A$. Denote these constant values by $p_i(A)$ and $q_i(A)$.
By \cite[p.18]{Krai}, the endpoints of the interval $A$ are
$p_n(A)/q_n(A)$ and $(p_n(A)+p_{n-1}(A))/(q_n(A)+q_{n-1}(A))$.
As a consequence, 
$$\lambda(A)=\frac{1}{q_n(A)(q_n(A)+q_{n-1}(A))}.$$
Since $q_n(A)>q_{n-1}(A)$ we obtain the desired double inequalities.
\end{proof}

The next lemma used to control the nonlinearity of $T$ 
can be proved by elementary calculations and hence omitted. See e.g., \cite[p.253 Claim]{FieFieYur02}
for details.
\begin{lemma}\label{distortion}
For every integer $n\geq1$ and every $A\in \mathscr{A}^n$,
$$\sup_{x,y\in A}\frac{DT^n(x)}{DT^n(y)}\leq e^{16}.$$
\end{lemma}
 
 %\subsection{Rate function}
Write $\phi=-\log|DT|$ and denote by
$\mathcal M_\phi(T)$ the set of $T$-invariant Borel probability measures on $(0,1)$
for which $\phi$ is integrable.
 For each $\mu\in\mathcal M_\phi(T)$
 denote by $h(\mu)$ the Kolmogorov-Sina{\u\i} entropy of $\mu$ with respect to $T$,
 and define $\chi(\mu)=-\int\phi d\mu$.
 Put $F(\mu)=h(\mu)-\chi(\mu)$.
 It is known  \cite{Wal78} that $\chi(\mu)\geq\alpha_{\rm min}$ and $F(\mu)\leq0$.

%Define a function $\alpha\in\left[\alpha_{\rm min},\infty\right)\mapsto \tilde I(\alpha)$ by
%$$\tilde I(\alpha)=\inf\{-F(\mu)\colon\mu\in\mathcal M_\phi(T), \ \chi(\mu)=\alpha\}.$$
%By definition, $\tilde I$ is non-negative, finite, convex. % and lower semi-continuous.

%The fact $b(\alpha)\to0$ as $\alpha\to\alpha_{\rm min}$ implies
%$\tilde I(\alpha)\to \alpha_{\rm min}$. $b(\alpha_{\rm min})=0$ implies
%$I\tilde(\alpha_{\rm min})=\alpha_{\rm min}$.

%Since $T$ satisfies R\'enyi's condition and $T^2$ is uniformly expanding,
% $\phi$ is H\"older continuous. 

%\begin{lemma}\label{rateid}
%For every set $A\subset\mathbb R$,
%$$\inf\left\{-F(\mu)\colon\mu\in\mathcal M_{\phi}(T), \chi(\mu)\in A\right\}
%\geq\inf_A \tilde I,$$
%and the equality holds if $A$ is an open set.
%\end{lemma}
%\begin{proof}
%For every $\alpha\in\mathbb R$,
%the definition of $J$ immediately yields 
%$$\inf\left\{-F(\mu)\colon\mu\in\mathcal M_{\phi}(T), \chi(\mu)=\alpha\right\}\geq \tilde I(\alpha),$$
%from which the desired inequality follows.
%If $A$ is a non-empty open set, then for each $\alpha\in A$ choosing $\epsilon>0$ such that 
%$(\alpha-\epsilon,\alpha+\epsilon)\subset A$ we have
 %\begin{align*}\tilde I(\alpha)&\geq \inf\left\{-F(\mu)\colon\mu\in\mathcal M_{\phi}(T),\ 
%\left|\chi(\mu)-\alpha\right|<\epsilon\right\}\\
%&\geq \inf\left\{-F(\mu)\colon\mu\in\mathcal M_{\phi}(T),\ \chi(\mu)\in A\right\}.\end{align*}
%Taking the infimum over all $\alpha\in A$ yields the reverse inequality.
%\end{proof}

\begin{lemma}\label{ratefunction}
For every $\alpha\in\left[\alpha_{\rm min},\infty\right)$,
$$I(\alpha)=\inf\{-F(\mu)\colon\mu\in\mathcal M_\phi(T), \ \chi(\mu)=\alpha\}.$$
In particular, $I$ is convex.
\end{lemma}

\begin{proof}
Denote the infimum by $\tilde I(\alpha)$.
 Choose a sequence $\{\nu_n\}$ in $\mathcal M_\phi(T)$
with $\chi(\nu_n)=\alpha$ and $\lim h(\nu_n)/\chi(\nu_n)= b(\alpha)$.
Then
$\tilde I(\alpha)\leq -\lim F(\nu_n)=
 I(\alpha).$
To show the reverse inequality,
choose a sequence $\{\mu_n\}$ in $\mathcal M_\phi(T)$ with
$\chi(\mu_n)=\alpha$ and $-F(\mu_n)\to \tilde I(\alpha)$ as $n\to\infty$.
Fix a measure $\nu\in\mathcal M_\phi(T)$ with
  $\chi(\nu)<\alpha$.
For each $n$ large enough fix $p_n\in(0,1]$ with $\chi(p_n\mu_n+(1-p_n)\nu)=\alpha.$
Then $\lim p_n\to1$ and hence $\lim h(p_n\mu_n+(1-p_n)\nu)=  \alpha-\tilde I(\alpha).$
 The variational formula in \cite{PolWei99} gives 
 $$b(\alpha)=\sup\left\{\frac{h(\mu)}{\chi(\mu)}\colon\mu\in\mathcal M_\phi(T),
 \ \chi(\mu)=\alpha\right\},$$
and therefore
$\alpha^{-1}(\alpha-\tilde I(\alpha))\leq b(\alpha)$, namely $I(\alpha)\leq \tilde I(\alpha)$ as required.
The convexity of $I$ is a consequence of the affinity of entropy and Lyapunov exponent on measures.
\end{proof}

\section{Upper bound with best exponential rate}
We are in position to prove the Main Theorem.

 \begin{lemma}\label{horse}
 Let $n\geq1$ be an integer and let $\alpha>0$.
Let $\mathscr{B}^n(\alpha)$ be a non-empty finite subset of $\{A\in\mathscr{A}^n\colon (2/n)\log q_n(A)\geq\alpha\}$.
There exists a measure $\mu\in\mathcal M_\phi(T)$ such that 
$$ 
 \lambda[\mathscr{B}^n(\alpha)]  \leq e^{16}e^{F(\mu)n}\ \ \text{ and }
 \ \  \chi(\mu)\geq\alpha-\frac{16}{n}.$$
% where $c>0$ is a uniform constant independent of $n$.
 \end{lemma}

 \begin{proof}
  Put $\widehat T=T^n$ and 
 $\Lambda=\bigcap_{m=0}^\infty{\widehat T}^{-m}[\mathscr{B}^n(\alpha)].$
 Then $\Lambda$ is a compact set and $\widehat T|_\Lambda\colon \Lambda\to \Lambda$ is continuous.
 Put $\widehat\phi=-\log|D\widehat T|$ and
 fix $y_0\in \Lambda$. 
% There exists a constant $C>0$ such that 
Lemma \ref{distortion} implies $ \sum_{i=0}^{m-1}(\widehat
 \phi(\widehat T^i(x))-\widehat
 \phi(\widehat T^i(y)))\leq 16$ for every $m\geq1$, every $x,y\in \Lambda$
 such that $\widehat T^i(x),\widehat T^i(y)$ belong to the same element of $\mathscr{B}^n(\alpha)$
 for each $i=0,\ldots,m-1$.
 %For each $v,w\in D^n$ with $v\neq w$ put 
%$\epsilon_{v,w}=\frac{1}{2}\inf_{x\in[v],y\in [w]}d_r(x,y)$ and $\epsilon=\inf_{v,w\in D^n}\epsilon_{v,w}$.
 %For each $m\geq1$ the set $(\widehat\sigma|_K)^{-m}(y_0)$ is an $(m,\epsilon)$-separated set
 %of $\widehat\sigma|_K$. 
% The definition of the topological pressure based on separated sets gives
 The variational principle \cite[Lemma 1.20]{Bow75} gives
 $$\sup_{\widehat\nu\in\mathcal M(\widehat T|_\Lambda)}\left(h_{\widehat T|_\Lambda}(\widehat\nu)+\int\widehat\phi d\widehat\nu\right)=
 \lim_{m\to\infty}\frac{1}{m}\log\left(\sum_{x\in (\widehat T|_\Lambda)^{-m}(y_0)}
 \exp{\sum_{i=0}^{m-1}\widehat
 \phi(\widehat T^i(x))}\right),$$
 with $\mathcal M(\widehat T|_\Lambda)$ the space of $\widehat T|_\Lambda$-invariant Borel probability measures
 endowed with the weak*-topology
 and $h_{\widehat T|_\Lambda}(\widehat\nu)$ the entropy of $\widehat\nu\in \mathcal M(\widehat T|_\Lambda)$
 with respect to $\widehat T|_\Lambda$.
By Lemma \ref{distortion},  
$\inf_{[w]}e^{\widehat\phi}\geq e^{-16} \lambda[w]$
holds for every $w\in D^n$.
Hence \begin{align*}\sum_{x\in (\widehat T|_\Lambda)^{-m}(y_0)}\exp\left(\sum_{i=0}^{m-1}\widehat\phi(\widehat T^i(x))\right)
&\geq\left(\inf_{y'\in \Lambda}\sum_{x\in (\widehat T|_\Lambda)^{-1}(y')}e^{\widehat\phi(x)}\right)^m\\
&\geq\left(e^{-16}\lambda[\mathscr{B}^n(\alpha)]
\right)^m.\end{align*}
Taking logs of both sides, dividing by $m$ and plugging the result into the previous inequality gives
$$\lim_{m\to\infty}\frac{1}{m}\log\left(\sum_{x\in (\widehat T|_\Lambda)^{-m}(y_0)}
\exp\sum_{i=0}^{m-1}\widehat\phi(\widehat T^i(x))\right)
\geq\log\lambda[\mathscr{B}^n(\alpha)]-16.$$
Plugging this into the previous inequality yields
$$\sup_{\widehat\nu\in\mathcal M(\widehat T|_\Lambda)}\left(h_{\widehat T|_\Lambda}(\widehat\nu)+\int\widehat\phi d\widehat\nu\right)\geq
\log\lambda[\mathscr{B}^n(\alpha)]-16.$$
Since $\mathcal M(\widehat T|_\Lambda)$ is compact 
and $\mathcal M(\widehat T|_\Lambda)\ni\widehat\nu\mapsto h_{\widehat T|_\Lambda}(\widehat\nu)+\int\widehat\phi d\widehat\nu $
is upper semi-continuous, there exists a measure $\widehat\mu\in \mathcal M(\widehat T|_\Lambda)$
which attains this supremum. The measure
$\mu = (1/n)\sum_{i=0}^{n-1}\widehat\mu\circ T^{-i}$ is in $\mathcal M_\phi(T)$.
From the second inequality in Lemma \ref{length} and Lemma \ref{distortion}, 
$\inf_{[\mathscr{B}^n(\alpha,c)]}\log |D\widehat T|\geq\alpha n-16$ holds.
Hence
$\chi(\mu)=(1/n)\int\log|D\widehat T|d\widehat\mu\geq\alpha-16/n$ as required.
 \end{proof}

% For each integer $n\geq1$ 
%put $$H^n=\left\{w\in \mathscr{A}^n\colon[w]\cap  \left\{x\in (0,1)\colon \frac{2}{n}\log q_n(x)\geq \alpha \right\}\neq\emptyset\right\}.$$ 

\begin{proof}[Proof of the Main Theorem]
Let $n\geq1$ be an integer.
We concentrate on the case $\alpha\in(2\gamma+\frac{16}{n},\infty)$ since the case 
$\alpha\in(\alpha_{\rm min},2\gamma-\frac{16}{n})$
is identical with the obvious modifications of statements.
Denote by $\lambda_n$ the distribution of $(2/n)\log q_n$.
For each $c>1$
choose a finite subset $\mathscr{B}^n(\alpha,c)$ of $\mathscr{A}^n$ such that
$\lambda_n([\alpha,\infty))\leq c\lambda[\mathscr{B}^n(\alpha,c)].$
By Lemma \ref{horse} there exists $\mu\in\mathcal M_\phi(T)$ which satisfies
$\lambda[\mathscr{B}^n(\alpha,c)]\leq e^{16}\exp(F(\mu)n)$ and
$\chi(\mu)\geq\alpha-16/n.$
Therefore
\begin{align*}
\lambda_n([\alpha,\infty))&\leq ce^{16}\exp(F(\mu)n)\\
&\leq ce^{16}\exp\left(\sup\left\{F(\mu)\colon\mu\in\mathcal M_\phi(T),\ \chi(\mu)\geq\alpha-\frac{16}{n}\right\}\right)\\
&\leq ce^{16}e^{-\inf_{\beta\in[\alpha-16/n,\infty)} I(\beta)}\\
&= ce^{16} e^{-I(\alpha-16/n)n}\\
&\leq ce^{16(I'(\alpha)+1)}e^{-I(\alpha)n}.
\end{align*}
%The third inequality follows from Lemma \ref{rateid}. 
For the last inequality we have used the convexity 
and the smoothness of $I$.
Since $c>1$ is arbitrary, we obtain 
$\lambda_n([\alpha,\infty))\leq C_\alpha e^{-I(\alpha)n}$ as required.
%From the definition, $I(\alpha_{\rm min})<\infty$.
%The continuity of $I$ at $\alpha=\alpha_{\rm min}$ follows from the convexity
%and lower semi-continuity of $I$.
\end{proof}

\subsection*{Acknowledgments}
This research was partially supported by
the Grant-in-Aid for Young Scientists (A) of the JSPS 15H05435 and
 the Grant-in-Aid for Scientific Research (B) of the JSPS 16KT0021.
 %and
 %the JSPS Core-to-Core Program ``Foundation
%of a Global Research Cooperative Center in Mathematics focused on Number Theory and Geometry''.

\end{document}